\documentclass[11pt]{article}
\usepackage[dvips]{graphicx}
\usepackage{epsfig,epsf}%,psfig}
\usepackage{sublabel}
\usepackage{latexsym}
\usepackage{amsthm}
\usepackage{amssymb}
\usepackage{graphicx}
\usepackage{amsmath}
\usepackage{pstricks}
\usepackage{fullpage}
\newpsobject{showgrid}{psgrid}{subgriddiv=1,griddots=10,gridlabels=6pt}

\newcommand {\norm} [1] { \lVert #1 \rVert}

\newcommand {\abs} [1] {\left| #1 \right|}
\newcommand {\Abs} [1] {\bigl\lvert #1 \bigr\rvert}

\newcommand {\floor}[1]{\lfloor #1\rfloor}

\newtheorem {thm} {Theorem}[section]

\newtheorem {lem} [thm] {Lemma}

\newtheorem {defn}{Definition}[section]
\newtheorem {qstn}{Question}[section]
\numberwithin{equation}{section}

\begin{document}
\title{Complex Singularities and the Lorenz Attractor}

\author{Divakar Viswanath\thanks{Department of Mathematics,
University of Michigan, 530 Church Street, Ann Arbor, MI 48109.
Partly supported by NSF grants
DMS-0407110 and DMS-0715510.} \and S\"{o}nmez \c{S}ahuto\u{g}lu\thanks{Department of Mathematics,
University of Toledo,  Toledo, OH 43606. Partly
supported by NSF grant DMS-0602191.}
}

\label{firstpage}

\maketitle

\begin{abstract}
The Lorenz attractor is one of the best known examples of applied
mathematics. However, much of what is known about it is a result of
numerical calculations and not of mathematical analysis. As a step
toward mathematical analysis, we allow the time variable in the three
dimensional Lorenz system to be complex, hoping that solutions that
have resisted analysis on the real line will give up their secrets in
the complex plane. Knowledge of singularities being fundamental to any
investigation in the complex plane, we build upon earlier work and
give a complete and consistent formal development of complex
singularities of the Lorenz system using {\it psi series}.  The psi
series contain two undetermined constants. In addition, the location
of the singularity is undetermined as a consequence of the autonomous
nature of the Lorenz system. We prove that the psi series converge,
using a technique that is simpler and more powerful than that of
Hille, thus implying a two-parameter family of singular solutions of
the Lorenz system. We pose three questions, answers to which may bring
us closer to understanding the connection of complex singularities to
Lorenz dynamics.\newline
\newline
\noindent{\bf AMS:} 34M35, 37D45.
\end{abstract}

%\begin{keywords} 
%\end{keywords}

%\begin{AMS}
%\end{AMS}

\section{Introduction}

The nonlinear system of equations
\begin{align}
\frac{dx}{dt} &= 10(y-x) \nonumber\\
\frac{dy}{dt} &= 28 x - y -xz \nonumber\\
\frac{dz}{dt} &= -8z/3 + xy, 
\label{eqn-1-1}
\end{align}
which is named after Lorenz, gives the best known example of a strange
attractor. Lorenz \cite{Lorenz1963, Lorenz1993} derived this system to
argue that the unpredictability of weather is due to the nature of the
solutions of the Navier-Stokes equations and not due to stochastic
terms of unknown origin, his point being that a deterministic system
could possess an attracting and invariant set on which the dynamics is
bounded and linearly unstable. When such strange attractors exist,
trajectories are chaotic and appear random.

 While Lorenz
\cite[p.141, 1963]{Lorenz1963} could write that the atmosphere
was not normally regarded as deterministic, we now know that the
incompressible Navier-Stokes equations by themselves explain a
remarkable wealth of turbulence phenomena including coherent motions
in the near-wall region, the law of the wall, intermittency, and
vortex structures in fully developed turbulence \cite{Davidson2004}.
The density and temperature of the atmosphere vary with altitude,
and there is significant electrical activity in the atmosphere
that is sustained by about $40,\!000$ thunderstorms that occur
around the world in any single day 
\cite[Chapter 9]{FeynmanLeightonSands1970}. If we nevertheless think
that the physics of the atmosphere is deterministic, Lorenz and
his system are partly responsible.

Lorenz's point of view was dynamical. He viewed the state of
\eqref{eqn-1-1} as a point in $R^3$ and its solutions as trajectories
in $R^3$. The dynamical point of view has overwhelmingly dominated
work on the Lorenz system and Lorenz's original paper
\cite{Lorenz1963} has remained an outstanding introduction to
dynamics. In it, a careful reader can find discussions of numerical
errors, of concepts of stability, of symbolic dynamics (aspects of
which Lorenz seems to have rediscovered for himself), of the density
of periodic solutions on the Lorenz attractor, and of the fractal
nature of the Lorenz attractor.

The point of view in this paper, unlike Lorenz's, will be mainly
function theoretic. We view $t$ in \eqref{eqn-1-1} as a complex
variable and $x$, $y$, $z$ as analytic functions of a complex
variable. Our interest is in triples of analytic functions which
satisfy
\eqref{eqn-1-1}.  Our hope is that an investigation in the complex
plane will open a route to the mathematical analysis of the Lorenz
system.

For the most part, we deal with certain singular solutions of the
Lorenz system, which will be introduced momentarily.  As the right
hand side of the Lorenz system \eqref{eqn-1-1} is analytic, every
solution of the Lorenz system admits analytic continuation to the
complex plane. For some solutions, the analytic continuations have
singularities of the form we deal with, as indicated by numerical
results summarized in Section 5. In the second part of this
introduction, we pose three questions to help connect the complex
singularities with Lorenz dynamics.

From residue integration, the method of steepest descent, and the use
of deformation of contours to effect analytic continuation of certain
special functions, we know that knowledge of singularities is often
useful to investigations in the complex plane. This observation
explains our focus on singular solutions of the Lorenz system.

\subsection{Psi series solutions of the Lorenz system}

The most common types of singularities are poles, algebraic branch
points, and logarithmic branch points. The singularities of the
Lorenz system that we examine are of none these types, but are given
by psi series representations.

\begin{defn}
A {\it logarithmic psi series} centered at $t_0$ is a series
of the form $\sum_{n=-N}^\infty p_n(\eta) (t-t_0)^n$, where
$N$ is an  integer, $\eta = \log(b(t-t_0))$ and each $p_n$ is a polynomial
in $\eta$. In the definition of $\eta$, $b$ is a complex number with
$\abs{b} = 1$, with $b = \pm i$ often being convenient choices.
\label{defn-1}
\end{defn}

Throughout this paper, $\log$ will denote the principal branch of
$\log$.  The choice of the branch is ultimately immaterial but taking
$\eta = \log(-i (t-t_0))$ instead of $\eta = \log (t-t_0)$ leads to
more convenient branch cuts if $\Im(t_0) < 0$, as we explain in
Section 3. For a slightly different definition of logarithmic psi
series, along with definitions of psi series of other types, see
\cite[Chapter 7.1]{Hille1976}.  The only type of psi series that
arises in this paper is the type given by Definition \ref{defn-1}, and
by psi series we refer to that definition only.

The psi series of Definition \ref{defn-1} are like the Laurent series,
except that the coefficients are polynomials in $\eta$ instead of
being constants. For that reason, the psi series singularities were
called pseudopoles by Hille \cite{Hille1973}. Even though the
coefficients are polynomials in $\eta$, each nonzero term of the
logarithmic psi series dominates the following term in magnitude in
the limit $t\rightarrow t_0$.

In an intriguing and original pair of papers, Tabor and Weiss
\cite{TaborWeiss1981} and Levine and Tabor \cite{LevineTabor1988}
considered psi series solutions of the Lorenz system
\eqref{eqn-1-1}. The psi series they used were expressed as a double
sum. Below we give the psi series in a different form:
\begin{alignat}{2}
x(t) &= &\frac{P_{-1}(\eta)}{t-t_0} + P_0(\eta) + P_1(\eta) (t-t_0)
+ P_2(\eta) (t-t_0)^2 + \cdots& \nonumber\\
y(t) &= \frac{Q_{-2}(\eta)}{(t-t_0)^2} +
&\frac{Q_{-1}(\eta)}{t-t_0} + Q_0(\eta) + Q_1(\eta) (t-t_0)
+ Q_2(\eta) (t-t_0)^2 + \cdots& \nonumber\\
z(t) &= \frac{R_{-2}(\eta)}{(t-t_0)^2} +
&\frac{R_{-1}(\eta)}{t-t_0} + R_0(\eta) + R_1(\eta) (t-t_0)
+ R_2(\eta) (t-t_0)^2 + \cdots& 
\label{eqn-1-2}
\end{alignat}
Here the $P_i$, $Q_i$, and $R_i$ are polynomials in $\eta$
where 
$\eta = \log(b(t-t_0))$
as in Definition \ref{defn-1}.
As the Lorenz system is autonomous, $t_0$ is an arbitrary complex number.
The fact that the leading powers of $(t-t_0)$ in the three series in
\eqref{eqn-1-2} are $-1$, $-2$, and $-2$ may be guessed by substituting
poles $(t-t_0)^{-\alpha}$, $(t-t_0)^{-\beta}$, $(t-t_0)^{-\gamma}$ for
$x$, $y$, $z$ into the Lorenz system and then solving for $\alpha$,
$\beta$, $\gamma$ by matching the order of the left and right hand
sides \cite{TaborWeiss1981}.  This test-power method
\cite[p. 90]{Hille1976} does not always work and can be tricked into
failing for the Lorenz system with a linear change of variables.

Melkonian and Zypchen \cite{MelkonianZypchen1995} have recast the psi
series of Tabor and Weiss \cite{TaborWeiss1981} into the formalism of
Hille
\cite{Hille1973}. The formal development of psi series that we give in
Section 3 is similar to that of Melkonian and Zypchen
\cite{MelkonianZypchen1995}, 
but improves that of Melkonian and Zypchen 
in two respects.  Firstly, the development in Section 3 shows the
dependence on undetermined constants $C$ and $D$ explicitly, pointing
out the occurrence of $\eta$ and $C$ in the group $(\eta +
C)$. Secondly, we prove that the degrees of $P_{m+1}$, $Q_m$, $R_m$
are given by $\floor{\frac{m+2}{2}}$ for $m=0,1,\ldots$. The proof
hinges on a surprising cancellation for $m=2$. It is important to get
such details fully right if a mathematical theory is to be set up. As
Hille
\cite[p. 68]{Hille1976} pointed out, ``constants of integration
play a remarkable role in the advanced theory of nonlinear DEs.''
In addition, a complete formal calculation is
essential for a fully correct convergence proof.

\begin{table}
\renewcommand{\arraystretch}{1.25}
\begin{center}
\begin{tabular}{c|c|c|c}
$\quad\quad$ $Q_{-2}$, $R_{-2}$ & & $-\frac{1}{5}\,i$& $-\frac{1}{5}$\\\hline
$P_{-1}$, $Q_{-1}$, $R_{-1}$ & $2\,i$  & $2\,i$ & ${\frac {17}{9}}$\\\hline
$P_{0}$, $Q_{0}$, $R_{0}$ & ${\frac {71}{9}}\,i$ & $-{\frac {349}{81}}\,i-{\frac {988}{81}}\,i(\eta+C)$ & ${\frac {1385}{54}}-{\frac {988}{81}}\,(\eta+C)$\\\hline
$P_{1}$, $Q_{1}$, $R_{1}$ & $-{\frac {9880}{81}}\,i(\eta+C)$ & $-{\frac {25991}{108}}\,i+{\frac {64220}{243}}\,i(\eta+C)$ & $-{\frac {211189}{972}}+{\frac {167960}{729}}\,(\eta+C)$\\\hline
$P_{2}$ &  \multicolumn{3}{c}{$-{\frac {2108195}{972}}\,i+{\frac {469300}{243}}\,i(\eta+C)$}
\\\hline
$Q_{2}$ & \multicolumn{3}{c}{$\frac{3}{10}\,i D-{\frac {477319147}{131220}}\,i-{\frac {167831753}{65610}}\,i(\eta+C)-{\frac 
{273676}{2187}}\,i{(\eta+C)}^{2}
$}\\\hline
$R_{2}$ & \multicolumn{3}{c}{$-\frac{1}{5}\, D +{\frac {138959125}{17496}}-{\frac {58846039}{32805}}\,(\eta+C)-{\frac {
1444456}{2187}}\,{(\eta+C)}^{2}$}\\\hline
$P_{3}$ & \multicolumn{3}{c}{$i D-{\frac {96356411}{6561}}\,i(\eta+C)-{\frac {2736760}{6561}}\,i{(\eta+C)}^{2}$}\\\hline
$Q_{3}$ & \multicolumn{3}{c}{$-{\frac {25925844899}{708588}}\,i+{\frac {32}{27}}\, i D-{\frac {
516846814}{59049}}\,i(\eta+C)+{\frac {26636480}{2187}}\,i{(\eta+C)}^{2}
$}\\\hline
$R_{3}$ & \multicolumn{3}{c}{$-{\frac {55}{27}}  D  +{\frac {64036692917}{3542940}}-
{\frac {2458513}{2187}}(\eta+C)+{\frac {813193160}{59049}}{(\eta+C)}^{2}$}\\\hline
$P_{4}$ & \multicolumn{3}{c}{${\frac {25}{54}}\, i D-{\frac {64653009635}{708588}}\,i-{\frac {
107735075}{118098}}\,i(\eta+C)+{\frac {206615500}{6561}}\,i{(\eta+C)}^{2}
$}
\end{tabular}
\end{center}
\caption[xyz]{Coefficients of the psi series of \eqref{eqn-1-2},
with $\eta$ as in Definition \ref{defn-1}.  Evidently, the degrees of
$P_{m+1}$, $Q_m$, $R_m$ are $\floor{\frac{m+2}{2}}$ for $m=0, 1, 2,
3$. The other valid choice for coefficients of \eqref{eqn-1-2} is
obtained by changing the signs of all the $P_i$s and $Q_i$s, while
leaving the $R_i$s unchanged. The constants $C$ and $D$ are both undetermined.}
\label{table-1}
\end{table}

The first few coefficients of the psi series \eqref{eqn-1-2} are
listed in Table \ref{table-1}. It is evident that $\eta$ and $C$
always occur in the group $(\eta+C)$. If $D$ were real, the
coefficients of the polynomials in $(\eta+C)$ listed in that table
would all be either pure imaginary or real.

The following is one of our main theorems. It reappears in a more
specific form in Section 4, where it is proved.

\begin{thm}
The psi series \eqref{eqn-1-2}, some of whose coefficients are listed
in Table \ref{table-1}, satisfy the Lorenz system \eqref{eqn-1-1} in
the disc $\abs{t-t_0} \leq r$ for some $r > 0$, but with the singular
point $t=t_0$ and a branch cut deleted from the disc. The constants
$C$ and $D$ are undetermined.
\label{thm-1-1}
\end{thm}

The proof of this theorem is valid for any choice of the undetermined
constants $C$ and $D$, but the estimate for $r$ depends upon the
choice. A key step in its proof is to show the convergence of the psi
series.

An important aspect of the convergence of the Lorenz psi
series is not brought out in Theorem \ref{thm-1-1}. As evident
from the appearance of $\eta$ in Definition
\ref{defn-1}, a typical psi series will have logarithmic branch points
in the $t$-plane. 
To get around the multiple-valuedness, Theorem \ref{thm-1-1} fixes
a branch cut in the $t$-plane. The branch cut can be dispensed with by 
parametrizing the Riemann surface using $\eta$. A discussion
of convergence in the $\eta$-plane is found in Section 4 (see
Figure \ref{fig-4} in particular).

Hille's ``frontal attack'' to prove convergence of psi series can be
modified to apply to the Lorenz system \cite{Hille1973,
MelkonianZypchen1995}. 
In an
appendix, Hille
\cite{Hille1974} pointed out that his technique could only handle the
Emden-Fowler system (see Section 2) with $p=2$, while a more
complicated technique due to Smith \cite{Smith1975} could handle
$p=2,3,\ldots$. The technique we use in Section 4.1 is also a frontal
attack, but it is a good deal more transparent than Hille's approach.
In place of an elaborate analytic set up and an inductive hypothesis
to bound the coefficients of the psi series, we use the Laplace
transform, elementary combinatorics, and an elementary implicit
function theorem. Our technique seems to extend to all the cases
handled by Smith
\cite{Smith1975}. Detailed comments on this point are found in Section
4.2.

\subsection{Complex singularities and Lorenz dynamics: three questions}
From Theorem \ref{thm-1-1} we get a two-parameter family of
singular solutions of the Lorenz system \eqref{eqn-1-1}. The form of
the singular solutions is given by the psi series \eqref{eqn-1-2} and
the two undetermined constants $C$ and $D$ are shown in Table
\ref{table-1}. The location $t_0$ of the singularity can be anywhere
in the complex $t$-plane. 

For some definite integrals, the
singularities of the integrand and Cauchy's residue theorem imply the
value of the integral. So we ask, what do the singular solutions of the
Lorenz system tell us about the dynamics in $R^3$ for real time? As 
the analytic theory of solutions of the Lorenz system is still in its
infancy, a complete answer to the question cannot be given. Nevertheless,
the question merits a thorough discussion.

\begin{figure}
\begin{center}
\includegraphics[scale=0.3]{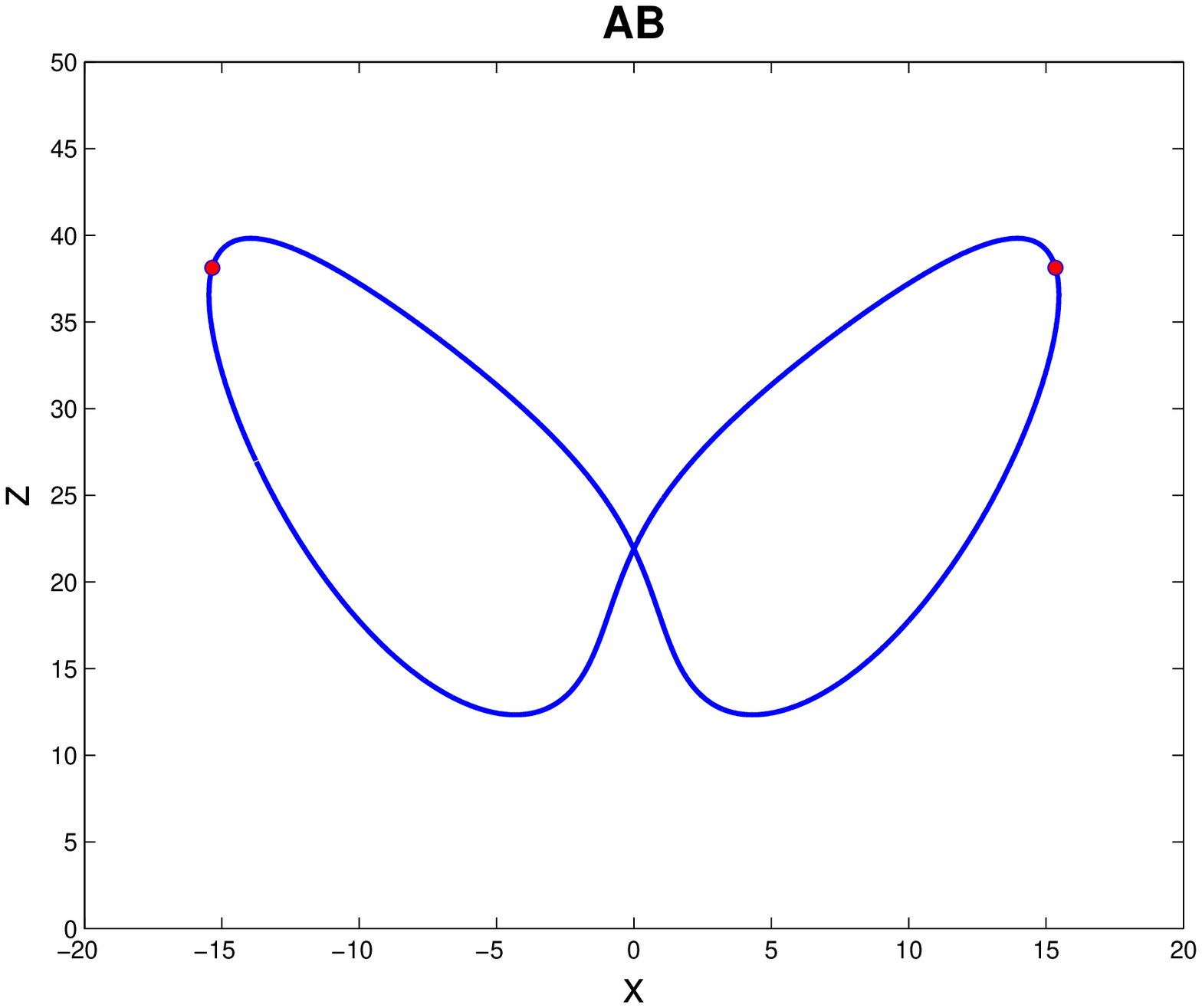}
\includegraphics[scale=0.3]{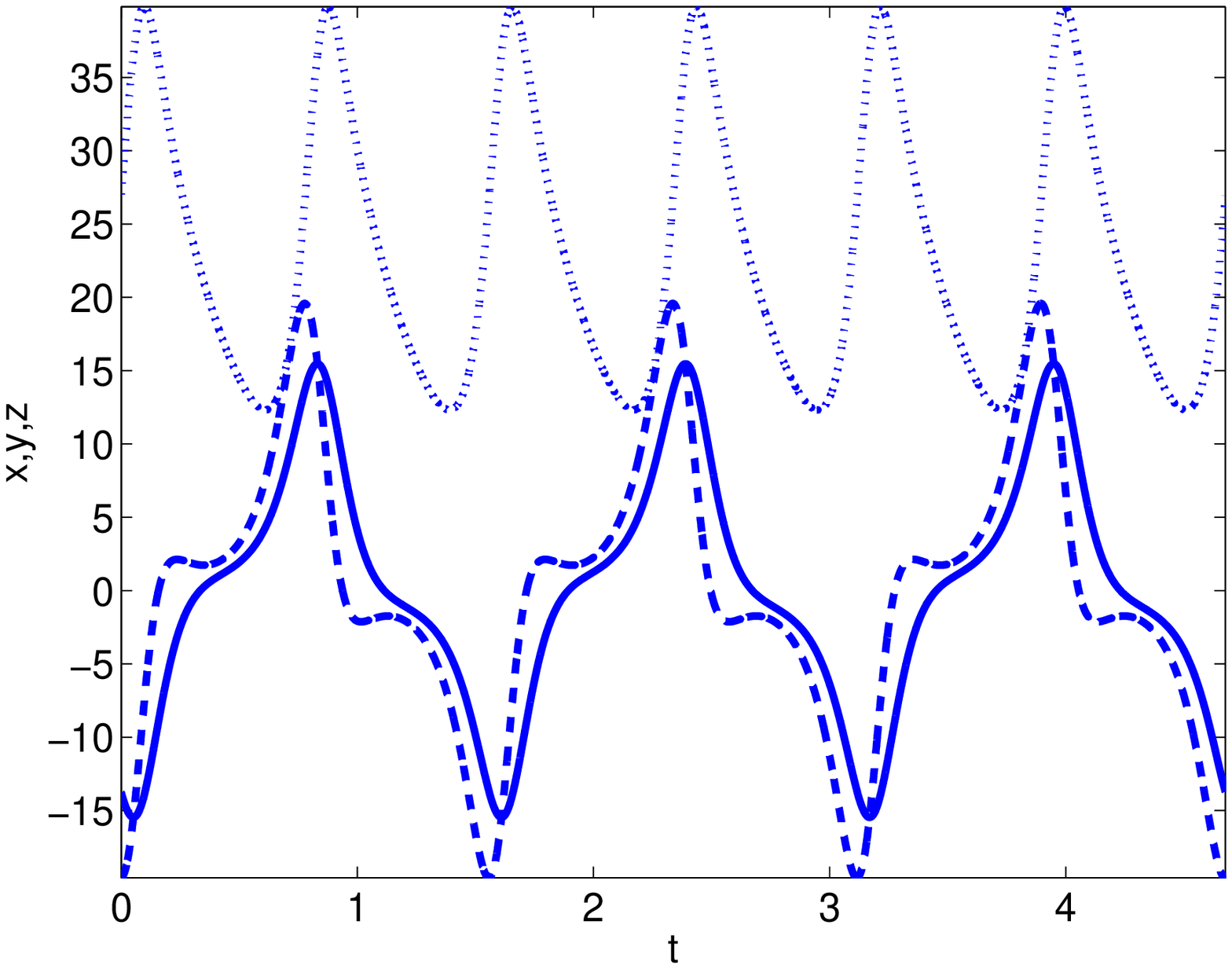}
\includegraphics[scale=0.3]{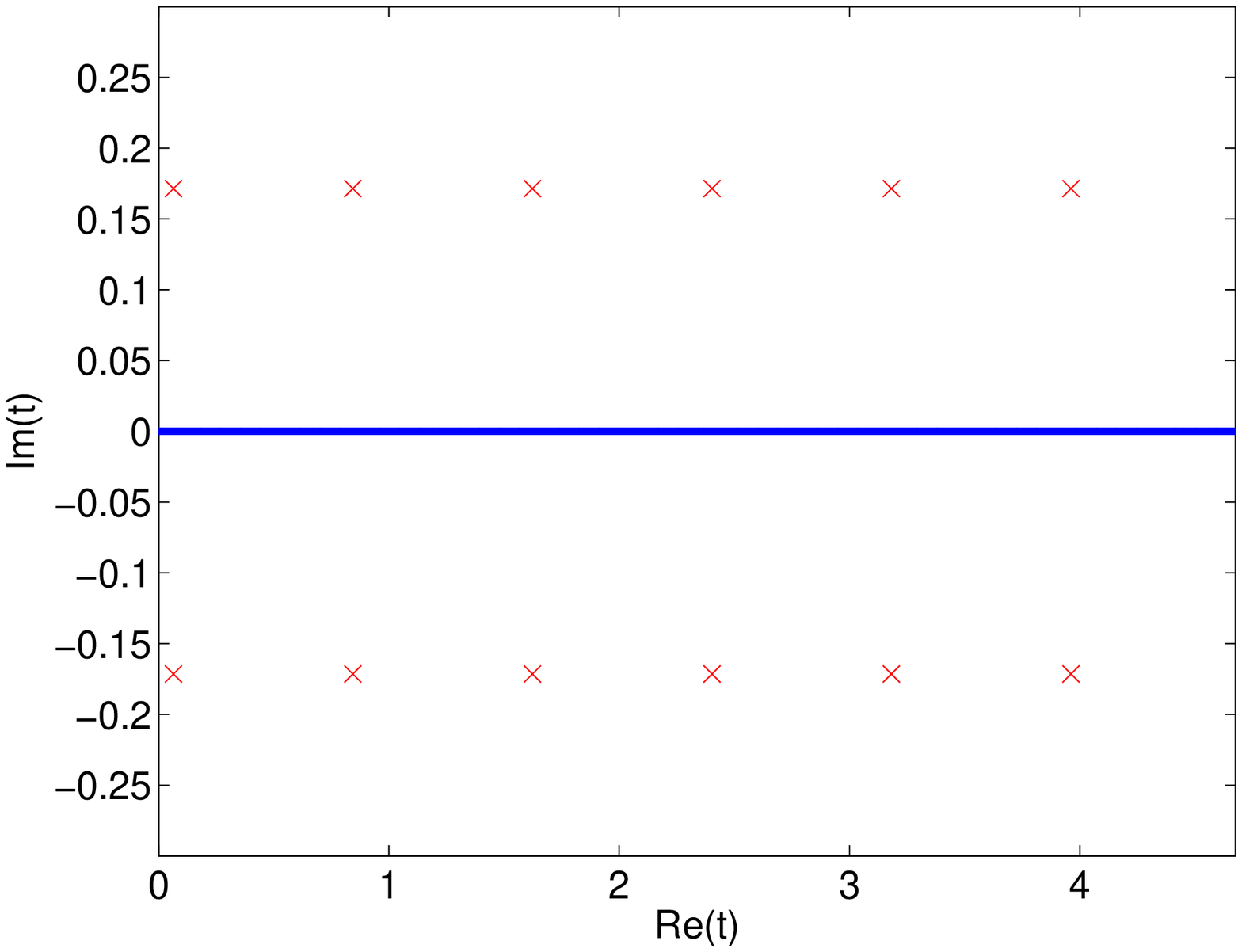}
\end{center}
\caption[xyz]{The periodic orbit in the first plot is labeled $AB$ to
indicate the sequence in which it moves between the $A$
quadrant (with $x<-16.432$, $y<-16.432$, $z=27$) and the $B$ quadrant
(with $x>16.432$, $y>16.432$, $z=27$).
Each filled circle is directly below
a singularity in the complex $t$-plane. In the middle are plots
of $x(t)$ (solid), $y(t)$ (dashed), $z(t)$ (dotted)
against real $t$. In the rightmost plot, the location of the
complex singularities of $AB$ that are closest to the real line
are marked as crosses.The orbit $AB$ is computed with $547$
digits of precision.
}
\label{fig-1}
\end{figure}

Many beautiful visualizations of the Lorenz attractor are found on the
INTERNET. The visualizations originally offered by Lorenz
\cite{Lorenz1963} are packed with information and are models of
concision. The Lorenz attractor is a butterfly-like subset of
$R^3$. Except for the fixed points, all trajectories either approach
the attractor as $t\rightarrow \infty$ or are already on it.

Figure \ref{fig-1} shows the periodic orbit labeled $AB$, which
resides on the attractor.  A great advantage of computing such orbits,
as opposed to arbitrary trajectories, is that the computations take on
a definite character that makes it possible to report them
precisely. As already mentioned at the beginning of this introduction,
periodic orbits are believed to be dense in the Lorenz
attractor. Such orbits can be computed with great precision. The
locations of the complex singularities shown in the rightmost plot of
Figure \ref{fig-1} were obtained by computing the orbit $AB$ with more
than $500$ digits of precision.

\begin{figure}
\begin{center}
\includegraphics[height=1.25in, width=5.8in]{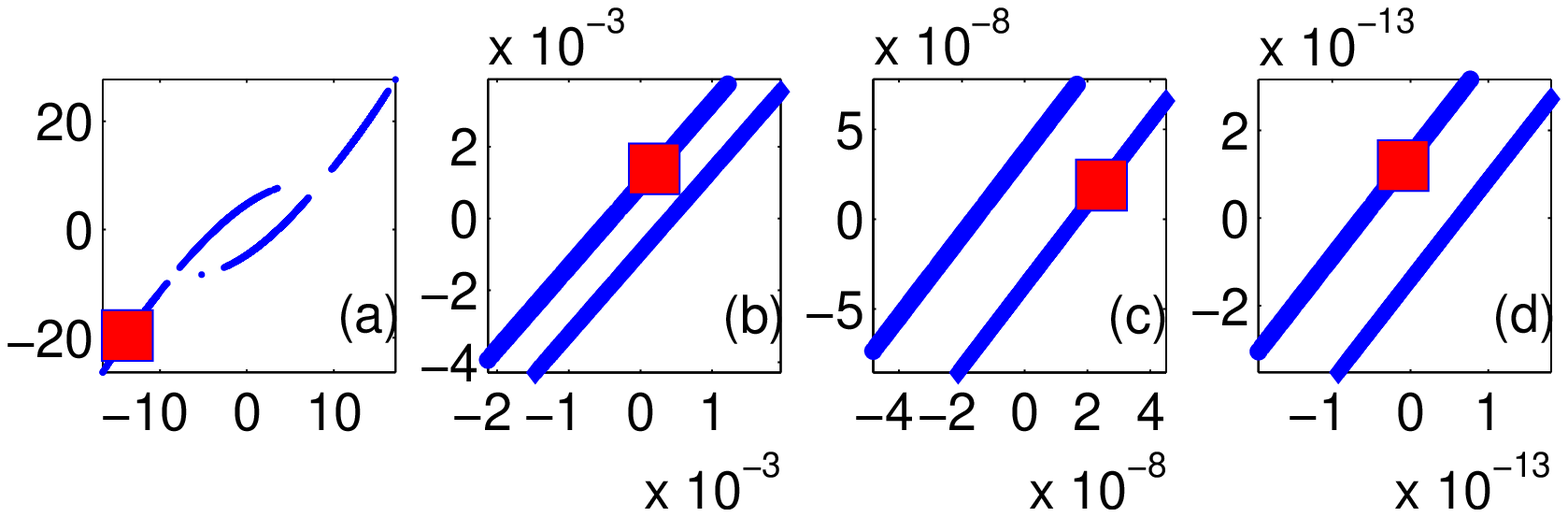}\\
\vspace*{.25in}
\includegraphics[height=1.25in, width=5.8in]{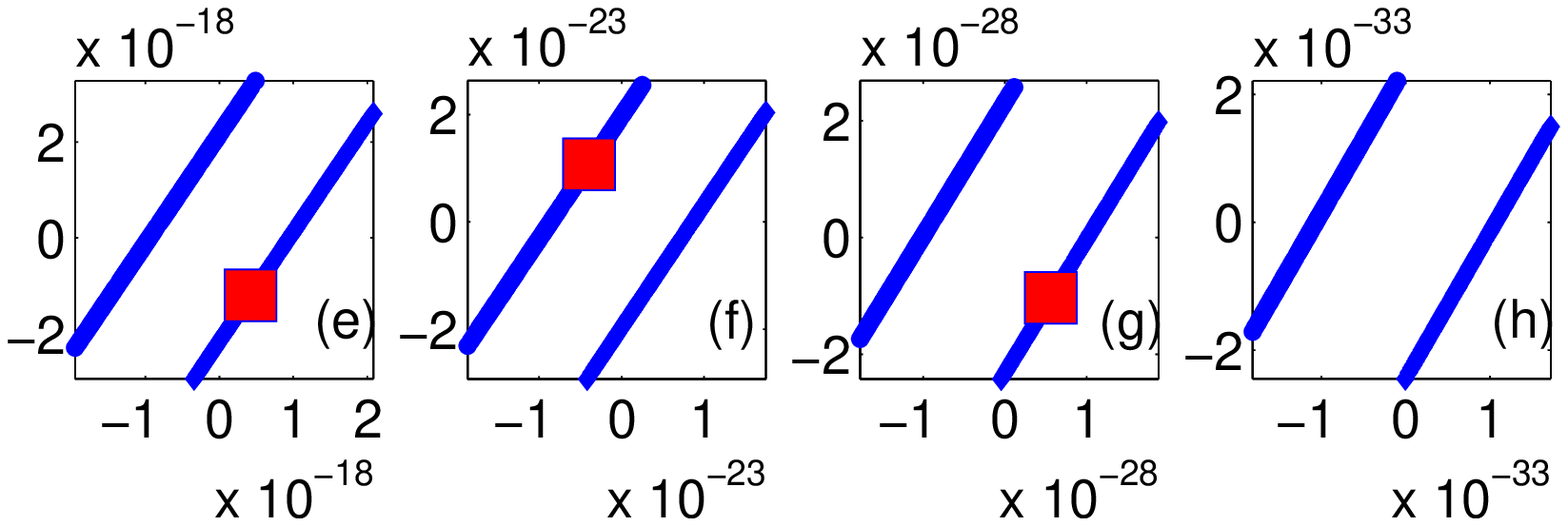}
\end{center}
\caption[xyz]{
Fractal property of the Lorenz attractor.  (a): The intersection of an
arbitrary trajectory on the Lorenz attractor with the section
$z=27$. The plot shows a rectangle in the $x$-$y$ plane.  All later plots ((b)
and above) zoom in on a tiny region (too small to be seen by the
unaided eye) at the center of the red rectangle of the preceding plot
to show that what appears to be a line is in fact not a line.  These
plots, and the plots of \cite{Viswanath2003,Viswanath2004} of which these plots
are a refinement, appear to be the only plots made of the fractal
structure of the Lorenz attractor.}
\label{fig-2}
\end{figure}

A worthy goal for the analytic theory of the Lorenz system is a proof
of existence of periodic solutions $(x(t), y(t), z(t))$ of the Lorenz system
\eqref{eqn-1-1}, where we seek a proof that is based solely on mental conceptions.
There is a definiteness to seeking periodic solutions as already
pointed out. In addition, periodic orbits are key to extracting order
from chaos, to borrow an expression from Strogatz \cite{Strogatz1994}.
For instance, Figure \ref{fig-2}, which illustrates the fractal
property of the Lorenz attractor, was obtained by computing periodic
orbits. The plots were computed in parallel on a machine with two
quadcore 2.33 GHz Xeon processors. The plots took a day or two of
computing. For the theory behind such computations, see
\cite{Viswanath2003} and \cite{Viswanath2004}.

A proof of existence of periodic solutions of the Lorenz system
\eqref{eqn-1-1} appears to be far away. We formulate three questions to
serve as more immediate goals for the development of the analytic
theory of the Lorenz system.

\begin{qstn}
Are all singular solutions of the Lorenz system given by psi series
expansions \eqref{eqn-1-2} with suitable choice of the undetermined
constants $C$ and $D$?
\label{qstn-1}
\end{qstn}

The role of the undetermined constants $C$ and $D$ is partly shown in
Table \ref{table-1}. Their role in the psi series is clarified further
in Sections 3 and 4. Lorenz \cite{Lorenz1963} gave arguments that
partially imply that a real solution of the Lorenz system cannot
become singular in finite time. The implication covers both increasing
and decreasing time. In Section 5, we give a complete proof of that
result.  Thus for solutions of the Lorenz system that are real for
real $t$, the locations $t_0$ of the complex singularities must have a
nonzero imaginary part. In fact, Foias and others
\cite[Theorem 2.3]{FoiasJollyKukavicaTiti2001} have proved that for
solutions {\it on} the Lorenz attractor the imaginary part of the
location of the singularity in the complex $t$-plane must exceed
$0.037$ in magnitude.  For an investigation of the backward in time
behavior of the Lorenz system (for real data), see the paper by Foias
and Jolly
\cite{FoiasJolly2005}.

The techniques used to deduce psi series solutions of the Lorenz
system are not of much use for answering Question
\ref{qstn-1}. However, if $t_0$ is {\it any} singular point of the Lorenz
system, then $\abs{x(t)} + \abs{y(t)} + \abs{z(t)}
\rightarrow \infty$ as $t\rightarrow t_0$, as implied by a slightly
stronger theorem proved in Section 5.

For analytic functions such as the gamma and zeta functions,
analytic continuation into the complex plane is an important
step in understanding the true nature of those functions
\cite{Olver1997}. The question of analytic continuation is important
in the theory of differential equations in the complex plane as well
\cite{Hille1976}. These observations motivate us to ask the following
question.

\begin{qstn}
Do solutions of the initial value problem for the Lorenz system 
with $(x(0), y(0), z(0))$ being finite (but possibly complex)
admit of analytic continuation to the entire complex $t$-plane
except for branch points?
\label{qstn-2}
\end{qstn}

An affirmative answer to Question \ref{qstn-1} appears to imply an
affirmative answer to Question \ref{qstn-2}. The process of analytic
continuation can be blocked by singularities. But if all singularities
are given by psi series of the form \eqref{eqn-1-2}, Theorem
\ref{thm-1-1} implies that we can continue around any such singularity
into a disc of finite radius around that singularity (radius is $r$ in
the theorem). The possibility where a succession of psi series
singular solutions of decreasing radii of convergence accumulate on
another singular point is easily ruled out, if the answer to Question
\ref{qstn-1} is yes.

Singular solutions given by psi series representations exist for
plane quadratic systems as well as plane polynomial systems
\cite{Hille1974, Smith1975}. Such planar systems certainly
cannot exhibit chaos \cite{Strogatz1994}. The dynamics of planar
systems is tightly circumscribed by results such as the
Poincar\'{e}-Bendixson theorem.  Unlike the Lorenz system, the planar
systems considered by Hille \cite{Hille1974} and Smith
\cite{Smith1975} can have real solutions that develop singularities
in finite time.  Yet one is probably justified in thinking the
mere existence of singular solutions represented by psi series is
unlikely to tell us anything about the chaotic nature of the Lorenz
system.

This is perhaps the place to comment on the three free parameters with
which the Lorenz system is usually written, but which are given the
values used by Lorenz \cite{Lorenz1963} in \eqref{eqn-1-1}.  The three
parameters correspond to the Rayleigh number, the Prandtl number, and
the system size for the convection PDE from which the Lorenz system
was derived. With regard to the choice of these parameters, there are
three cases for which the Lorenz system admits a Laurent series as a
solution \cite{Segur1982, TaborWeiss1981}. There are five other cases,
due to Segur \cite{Segur1982} and Ku\'{s} \cite{Kus1983}, for which
time-dependent integrals of motion are known.  In their pioneering
work, Tabor and Weiss \cite{TaborWeiss1981} considered the connection
between integrability and the type of the singularities. For another
discussion of the connection between psi series and integrability, see
\cite{DelshamsMir1997}.

In addition to the integrable cases, there are a number of other
regions in parameter space where the Lorenz system has non-chaotic
dynamics yet admits singular solutions with psi series representation.
In these instances, it is quite possible that even though the
real-valued dynamics is non-chaotic, more varied solutions
exist when complex numbers are allowed.  In the case of plane
polynomial systems, although the differential equations cannot have
chaotic solutions that are real
\cite{CoddingtonLevinson1955,Strogatz1994}, the equations may have chaotic
solutions that are complex.

It is not entirely clear how the nature of the singularity can be
connected to chaotic dynamics. It is perhaps significant that only
real solutions have a bearing on dynamics. Therefore we ask the
following question.

\begin{qstn}
If a psi series solution of the Lorenz system \eqref{eqn-1-1} of the
form
\eqref{eqn-1-2} is obtained by analytic continuation of a solution
that is real for real $t$, what constraints must $C$, $D$
and $t_0$ satisfy?
\label{qstn-3}
\end{qstn}

The detailed development of psi series found in Section 3 and partly
shown in Table \ref{table-1} could help answer this question.
Numerical computations are also likely to be useful.  A suspicion of
ours is that the undetermined constant $D$ is real for the psi series
singularities of Question \ref{qstn-3}.

\section{A brief history of early work on psi series}

The equation of Briot and Bouquet
\begin{equation}
t \frac{dw}{dt} = p t + w + F(t, w),
\label{eqn-2-1}
\end{equation}
where $F$ is a polynomial with quadratic and higher terms, seems to be
the simplest differential equation whose singularities are given by
psi series. Dulac \cite[p. 368, 1912]{Dulac1912} and Malmquist
\cite[p. 19, 1921]{Malmquist1921} (also see Theorem 11.3.1 of \cite{Hille1976})
proved that the general solution of \eqref{eqn-2-1} around $t=0$ is
given by a convergent psi series if $p$ is a positive integer. For 
generalizations to higher order Briot-Bouquet equations, 
see \cite{Iwano1963}.

In the last decade of his life, Einar Hille \cite{Hille1970,
Hille1973, Hille1974, Hille1976} interested himself in the
Emden-Fowler equation $d^2y/dt^2 = t^{-2/p} y^{1+2/p}$ with $p>1$
being a positive integer. The Emden-Fowler equation originally arose
in cosmology.  The special case $p=2$ is the Thomas-Fermi equation,
which arose in atomic physics. After sixty years of encounters with
differential equations, Hille wrote a splendid book on ordinary
differential equations in the complex plane \cite[1976]{Hille1976}.
The last chapter of that book gives an outline of the work of 
Hille and Russell A. Smith \cite{Smith1975} on psi series
singularities of the Emden-Fowler equation. The techniques involved
are highly relevant to the Lorenz system. In Section 4, we point out
that some of the theorems of Hille and Smith admit simpler proofs
using an approach introduced in that section.

 From Hille's illuminating bibliographic discussions
\cite{Hille1976}, it is clear that
Dulac
\cite[1934]{Dulac1934} was a central figure with regard to
psi series, with Horn \cite[1905]{Horn1905} being another early
contributor.  Hille does not mention Dulac's claim about one of the
Hilbert problems, however, and indeed that claim was mistaken
\cite{Ilyashenko2002}. It appears that the error was related to a
subtlety in the interpretation of psi series in the complex plane
\cite{Ilyashenko2002}.

\section{Formal development}
The formal development of psi series has a history that goes back a
hundred years or more.  All formal developments proceed in a similar
way---one begins with psi series and then determines their
coefficients using a recursion. 
In two of his papers, Hille
\cite{Hille1973, Hille1974}
gave clear and detailed formal developments.  Our derivation is quite
similar, but is more careful about subtleties such as the choice of
the branch of $\log$, the degrees of the polynomials $P_i$, $Q_i$ and
$R_i$ in
\eqref{eqn-1-2}, and  the role of the undetermined
constants ($C$ and $D$ in Table \ref{table-1}).

Since the Lorenz system \eqref{eqn-1-1} is autonomous, the
choice of the location $t_0$ of the singularity is arbitrary. 
For the sake of definiteness and because the primary interest is
in solutions that are real for real $t$, we may assume $\Re(t_0) < 0$.
and take $\eta = \log(-i(t-t_0))$ to obtain
a branch cut that does not intersect the real axis. However, nothing
changes if $t_0$ is arbitrary and some other branch cut is chosen
for defining $\eta$. The choice of branch cut is equivalent to the
choice of $b$ in Definition \ref{defn-1}.

The form of the singularity is
assumed to be given by \eqref{eqn-1-2}:
\begin{equation}
x(t) = \sum_{m=-1}^{\infty} P_m(\eta)(t-t_0)^m \quad
y(t) = \sum_{m=-2}^{\infty} Q_m(\eta)(t-t_0)^m \quad 
z(t) = \sum_{m=-2}^{\infty} R_m(\eta)(t-t_0)^m, 
\label{eqn-3-1}
\end{equation}
where the $P_m$, $Q_m$ and $R_m$ are polynomials in $\eta$. We arrived
at this form based on numerical work summarized in Section 5. However,
the credit for discovering the form of the psi series singularities of
the Lorenz system belongs for the most part 
to Tabor and Weiss \cite{TaborWeiss1981}.

Substituting \eqref{eqn-3-1} into \eqref{eqn-1-1} and denoting
derivatives with respect to $\eta$ by a prime, we get
\small
\begin{align}
\sublabon{equation}
\sum_{m=-1}^\infty (P'_m(\eta) + m P_m(\eta))(t-t_0)^{m-1} &=
10 Q_{-2}(t-t_0)^{-2} + \sum_{m=-1}^\infty(10Q_m(\eta) - 10 P_m(\eta))
(t-t_0)^m
\label{eqn-3-2a}\\
\sum_{m=-2}^\infty (Q'_m(\eta) + m Q_m(\eta))(t-t_0)^{m-1} &=
28\sum_{m=-1}^\infty P_m(\eta)(t-t_0)^m
-\sum_{m=-2}^\infty Q_m(\eta)(t-t_0)^m \nonumber\\
&\hspace*{.8in}-\sum_{m=-3}^\infty\Biggl(\sum_{j=-1}^{m+2}P_j(\eta)R_{m-j}(\eta)\Biggr)
(t-t_0)^m
\label{eqn-3-2b}\\
\sum_{m=-2}^\infty (R'_m(\eta) + m R_m(\eta))(t-t_0)^{m-1} &=
-\frac{8}{3}\sum_{m=-2}^\infty R_m(\eta) (t-t_0)^m\nonumber\\
&\hspace*{.8in}+\sum_{m=-3}^\infty\Biggl(\sum_{j=-1}^{m+2}P_j(\eta)Q_{m-j}(\eta)\Biggr)
(t-t_0)^m
\label{eqn-3-2c}
\end{align}
\sublaboff{equation}
\normalsize
For the psi series on either side of
(3.2), a nonzero term with $m=m_1$ is greater in magnitude 
than an $m=m_2$ term in the limit $t\rightarrow t_0$ if $m_1 <
m_2$. Therefore it is formally consistent to equate powers of
$(t-t_0)$ in increasing order.

Equating coefficients of $(t-t_0)^{-2}$ in \eqref{eqn-3-2a} and of
$(t-t_0)^{-3}$ in \eqref{eqn-3-2b} and \eqref{eqn-3-2c}, we get
$P'_{-1}-P_{-1} = 10 Q_{-2}$, $Q'_{-2}-2Q_{-2} = - P_{-1}R_{-2}$, and
$R_{-2}' - 2 R_{-2} = P_{-1}Q_{-2}$. The degree of $P_{-1}$ and $Q_{-2}$
in $\eta$ must be the same, while the degree of $R_{-2}$ must be
twice that degree and the degree of $Q_{-2}$ must be the sum
of the degrees of the other two. The only possibility is for 
all the degrees to be zero. We get
\begin{equation}
(P_{-1}, Q_{-2}, R_{-2}) = (2i, -i/5, -1/5) \quad\quad \text{or}\quad\quad
(-2i, i/5, -1/5).
\label{eqn-3-3}
\end{equation}
We consider only the first possibility for now, but will account for
the second possibility in Lemma \ref{lem-3-2}.

The next set of equations is $P'_0 = 10(Q_{-1}-P_{-1})$,
$Q'_{-1} = Q_{-1} - 2iR_{-1} + P_0/5 + i/5$, and 
$R'_{-1} = R_{-1} + 2iQ_{-1} - i P_0/5 + 8/15$. The only solution
polynomial in $\eta$ is given by
\begin{equation}
(P_0, Q_{-1}, R_{-1}) = (71i/9, 2i, 17/9).
\label{eqn-3-4}
\end{equation}

For $m=0,1,2,\ldots$, we equate powers of $(t-t_0)^m$ in \eqref{eqn-3-2a}
and powers of $(t-t_0)^{m-1}$ in \eqref{eqn-3-2b} and \eqref{eqn-3-2c}
to get,
\begin{equation}
X'_m = A_m X_m + F_m(\eta),
\label{eqn-3-5}
\end{equation}
where \small
\begin{equation}
X_m = \begin{pmatrix} P_{m+1} \\ Q_m\\R_m \end{pmatrix},\quad
A_m = \begin{pmatrix} -m-1 & 10 & 0\\
\frac{1}{5} & -m & -2i \\ -\frac{i}{5} & 2i & -m 
\end{pmatrix},\quad
F_m = \begin{pmatrix}
-10 P_m \\
28P_{m-1}-Q_{m-1}-\sum_{j=0}^m P_jR_{m-j-1}\\
-\frac{8}{3}R_{m-1} + \sum_{j=0}^m P_j Q_{m-j-1}
\end{pmatrix}.
\label{eqn-3-6}
\end{equation}
\normalsize
The eigenvalues of $A_m$ are $-m+2$, $-m$, and $-m-3$. If the
linear system \eqref{eqn-3-5} is diagonalized  using the eigenvectors
of $A_m$ as a basis, it turns into three scalar equations of
the form $d\xi/d\eta = \alpha \xi + f(\eta)$, with $\alpha$ being
$-m+2$ or $-m$ or $-m-3$ and with $f$ being a polynomial in each case.
If $\alpha \neq 0$, we have a unique polynomial solution for $\xi(\eta)$
whose degree is the same as that of $f$. 

We can have $\alpha = 0$ if and only if $m=0$ or $m=2$. Thus if
$m\neq 0$ and $m\neq 2$, we can assert that \eqref{eqn-3-5} has
a unique polynomial solution $X_m$ and the degree of that solution
in $\eta$ is the same as that of $F_m$.

In the case $m=0$, $F_m$ is a constant and the three scalar
equations are of the form $d\xi/d\eta =  2\xi + \beta_1$,
$d\xi/d\eta = -3\xi + \beta_2$ and $d\xi/d\eta = \beta_3$, where
the $\beta_i$ are known constants. The only admissible solution of
either of the first two equations is a constant. The last equation
however has the solution $\beta_3(\eta + C)$, where $C$ is an
undetermined constant. If the eigenvectors of $A_0$ are
multiplied by the respective solutions and summed, we get
\begin{equation}
\begin{pmatrix}
P_1\\Q_0\\R_0
\end{pmatrix}
=
\begin{pmatrix}
-9880i/81\\-988i/81\\-988/81
\end{pmatrix} (\eta+C)
+\begin{pmatrix}
0\\ -349i/81 \\ 1385/54
\end{pmatrix},
\end{equation}
where the factor multiplying $(\eta+C)$ is the eigenvector of $A_0$
that corresponds to the eigenvalue $-m=0$.

The matrix $A_m$ has a zero eigenvalue again when $m=2$. In this case,
the degree of $F_m$ in $\eta$ is $2$. We would expect the polynomial
solution $X_m$ of \eqref{eqn-3-6} to be cubic. However, the component
of $F_m$ along the eigenvector of $A_m$ corresponding to the
eigenvalue $-m+2=0$ is zero (with regard to this point compare (2.9)
of
\cite{TaborWeiss1981}). Therefore $P_3$, $Q_2$ and $R_2$, which
make up $X_2$, are all quadratic in $\eta$ as shown in Table
\ref{table-1}. 
A new undetermined constant $D$ enters at this stage.  If $P_3$,
$Q_2$ and $R_2$ were
cubic and not quadratic, $\floor{\frac{m+2}{2}}$ in the lemma below
would be replaced by $\floor{\frac{m+2}{2}} + \floor{\frac{m+2}{4}}$.

\begin{lem}
The degrees of the polynomials $P_{m+1}(\eta)$, $Q_m(\eta)$ and $R_m(\eta)$
are at most $\floor{\frac{m+2}{2}}$ for $m=0,1,2,\ldots$.
\label{lem-3-1}
\end{lem}
\begin{proof}
For $m=0,1,2$, the lemma can be verified explicitly using Table \ref{table-1}.
If the maximum degree of a component of $X_k$ is $d_k$  for $0\leq k < m$,
\eqref{eqn-3-4} and \eqref{eqn-3-6} imply that the degree of $F_m$
is at most 
$$\max_{0\leq j \leq m}(d_{j-1} + d_{m-j-1}),$$
where we assume $m\geq 3$ and take $d_{-1} = 0$. We use the inductive hypothesis
and note 
$$d_{j-1} + d_{m-j-1} \leq \floor{\frac{j+1}{2}} + \floor{\frac{m-j+1}{2}}
\leq \floor{\frac{m+2}{2}}$$
for $0\leq j\leq m$ to complete the proof. The second inequality
above is an equality for odd $j$. 
\end{proof}

It appears as if the degrees in Lemma \ref{lem-3-1} are actually equal
to $\floor{\frac{m+2}{2}}$. To prove as much, one has to rule out
cancellations that can happen in a variety of ways, which may or may
not be worth the trouble. Below we give a formula for the polynomial
solution $X_m$ of \eqref{eqn-3-5} that is easily derived using the
variation of constants formula and integration by parts:
\begin{equation}
X_m = - \sum_{j=0}^{\floor{\frac{m+2}{2}}}A_m^{-j-1}\frac{d^{j}F_m}{d\eta^j},
\label{eqn-3-8}
\end{equation}
for $m\geq 3$. The correctness of \eqref{eqn-3-8} can be verified
by direct substitution into \eqref{eqn-3-5}.

The lemma below summarizes the discussion in this section.

\begin{lem}
\begin{enumerate}
\item[(i)]
For the coefficients $P_{m+1}$, $Q_m$, $R_m$ shown in Table
\ref{table-1} for
$-2\leq m \leq 3$ and defined for $m\geq 3$ by \eqref{eqn-3-6} and
\eqref{eqn-3-8}, the psi series \eqref{eqn-3-1} (or \eqref{eqn-1-2})
satisfy the Lorenz system \eqref{eqn-1-1} formally. The location
of the singularity $t_0$ is arbitrary and two undetermined constants,
$C$ and $D$, occur in the psi series. The constant $C$ and $\eta$
always occur in the group $(\eta+C)$.
\item[(ii)]
Another formal solution is obtained by flipping the signs of all
the $P$s and the $Q$s while leaving the $R$s unchanged.
\item[(iii)]
For the solution to be formally valid, $\eta$ can be defined as
$\log(b (t-t_0))$ for any complex number $b$ with $\abs{b} = 1$.
\end{enumerate}
\label{lem-3-2}
\end{lem}
\begin{proof}
For the part about flipping signs, note that the Lorenz system is
unchanged by the transformation $(x,y,z)\rightarrow(-x,-y,z)$.
More specifically, note that flipping signs of the $P$s and the
$Q$s changes the sign of the first two components of $F_m$ in
\eqref{eqn-3-6} but not that of the third component. 

This other formal solution accounts for the second possibility
in \eqref{eqn-3-3}.

\end{proof}

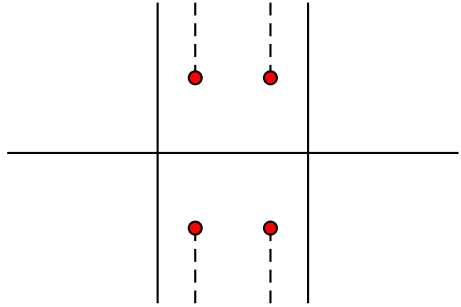
\begin{figure}
\begin{center}

\begin{pspicture}(3,-2.5)(11.0,2.5)
\psset{xunit=1cm, yunit=1cm}
%\showgrid
\psline(4,0)(10,0)
\psline(6,-2)(6,2)
\psline(8,-2)(8,2)
\psline[linestyle=dashed](6.5,1)(6.5,2)
\psline[linestyle=dashed](7.5,1)(7.5,2)
\psline[linestyle=dashed](6.5,-1)(6.5,-2)
\psline[linestyle=dashed](7.5,-1)(7.5,-2)
\pscircle[fillstyle=solid,fillcolor=red](6.5,1){.1}
\pscircle[fillstyle=solid,fillcolor=red](7.5,1){.1}
\pscircle[fillstyle=solid,fillcolor=red](6.5,-1){.1}
\pscircle[fillstyle=solid,fillcolor=red](7.5,-1){.1}
\end{pspicture}
\end{center}
\caption[xyz]{Schematic plot of the location of the singularities in
the $t$-plane for an orbit
such as $AB$. The singularities are shown as red spots and
the branch cuts are dashed. Only singularities within a single period
are shown in the $t$-plane (compare Figure \ref{fig-1}).}
\label{fig-3}
\end{figure}

If the psi series singularity is an analytic continuation of a
solution that is real for real $t$, the location $t_0$ of the
singularity must be off the real line (see Section 5). According as
$\Im(t_0) < 0$ or $\Im(t_0) > 0$, the choices $b = -i$ or $b = i$ 
give branch cuts that do not intersect the real
line, as shown in Figure \ref{fig-3}.

\section{Proof of convergence}
Hille's \cite{Hille1973} proof of the convergence of psi series
solutions relies on the formula $$X_m(\eta) = \int_{-\infty}^\eta
e^{(\eta-s)A_m} F_m(s) ds$$ for the solution $X_m$ of \eqref{eqn-3-5}
which is polynomial in $\eta$.  A similar formula is
fundamental to the approximation of strange attractors, including
Lorenz's, by algebraic sets in the work of Foias, Temam and others
\cite{FoiasJollyKukavicaTiti2001,FoiasTemam1988}.

Our proof of convergence does not use Hille's formula, but instead
relies on the Laplace transform and other devices. In the second part
of this section, we remark that our technique will likely give simpler
proofs for certain theorems of Hille and Smith. In one instance, our
technique can probably be used to prove a theorem that has been stated but not
proved completely.

\subsection{Psi series solutions of the Lorenz system}

If $p$ is a polynomial in $(\eta+C)$, we define $\abs{p}$ as the sum
of the absolute values of its coefficients.  Since $\eta$ and $C$
always occur in the group $(\eta+C)$, we can think of $C$ as being
subsumed by $\eta$.  $\abs{X_m}$ is defined as
the maximum of $\abs{P_{m+1}}$, $\abs{Q_m}$ and $\abs{R_m}$. For $m
\geq 2$, $\abs{X_m}$ will depend upon the undetermined constant $D$.
The key to the proof of convergence of the psi series \eqref{eqn-3-1}
is a bound of the form $\abs{X_m} < K_1 K_2^m$, where $K_1$ 
and $K_2$ are positive constants
 that depend upon the undetermined parameter
$D$.

For $F_m$ defined by \eqref{eqn-3-6}, $\abs{F_m}$ is the maximum of
$\abs{\cdot}$ over its three components, each of which is a polynomial
in $(\eta+C)$. We begin with the following easy lemma.

\begin{lem}
For $m\geq 3$,
$$\abs{F_m} \leq 30 \abs{X_{m-1}} + 28 \abs{X_{m-2}} +
\sum_{j=1}^{m-1} \abs{X_{m-j-1}} \abs{X_{j-1}}.$$
\label{lem-4-1}
\end{lem}
\begin{proof}
If $p$ and $q$ are polynomials  in $\eta+C$, $\abs{pq} \leq \abs{p}\abs{q}$
and $\abs{p+q} \leq \abs{p} + \abs{q}$. Repeated use of those inequalities
with the definition \eqref{eqn-3-6} of $X_m$ and $F_m$ gives
$$ 
\abs{F_m} \leq 10 \abs{X_{m-1}} + 28 \abs{X_{m-2}}
+\sum_{j=0}^{m} \abs{X_{m-j-1}} \abs{X_{j-1}}.
$$
The lemma results when the $j=0$ and $j=m$ terms are moved out of the summation
while using Table \ref{table-1} to note that $\abs{X_{-1}} < 10$.
\end{proof}

For matters related to the existence and uniqueness of the Laplace
transform that arise implicitly in the proof below, see
\cite{Widder1946}. In the lemma below, we only treat polynomials
in $\eta$ (assuming $C=0$), but the lemma still applies when
$\eta$ and $C$ occur in the group $(\eta+C)$ and $C\neq 0$.

\begin{lem}
Let $\alpha$ be a complex number with $\abs{\alpha} > 1$ and let
$f(\eta)$ be a polynomial in $\eta$. Let $\xi(\eta)$ be the polynomial
solution of the differential equation 
\begin{equation}
\frac{d\xi}{d\eta} = \alpha \xi + f(\eta).
\label{eqn-4-1}
\end{equation}
If the polynomial $f(\eta)$ is of degree $n$, assume $\abs{\alpha} 
\geq a (n+1/2)$ for some $a > 1$. Then
\begin{equation}
\abs{\xi} \leq \frac{1}{\abs{\alpha}}\frac{a}{a-1}\abs{f}.
\label{eqn-4-2}
\end{equation}
\label{lem-4-2}
\end{lem}
\begin{proof}
Let $f(\eta) = f_0 + f_1 \eta + \cdots + f_n \eta^n$. To take
the Laplace transform of \eqref{eqn-4-1}, we multiply \eqref{eqn-4-1}
by $e^{-\eta s}$ and integrate from $\eta=0$ to $\eta=\infty$. We get
$$
s \hat{\xi}(s) - \alpha \hat{\xi}(s) = \xi(0) + \frac{f_0}{s}
+\frac{1!f_1}{s^2} 
+\frac{2!f_2}{s^3}+
\cdots + \frac{n!f_n}{s^{n+1}}.
$$
Rearranging, we have
$$
\hat{\xi}(s) = \frac{\xi(0)}{s-\alpha} + \frac{f_0}{(s-\alpha)s}
+\frac{1!f_1}{(s-\alpha)s^2}+\cdots+\frac{n!f_n}{(s-\alpha)s^{n+1}}.
$$
All terms on the right hand side above except the first are rewritten
using the identity
$$
\frac{1}{(s-\alpha)s^k} 
= \frac{1}{\alpha^k(s-\alpha)} - \frac{1}{\alpha^k s}
-\frac{1}{\alpha^{k-1}s^2}- \cdots - \frac{1}{\alpha s^k}.
$$
In the resulting expression, $\xi(0)$ is chosen to cancel
all the $1/(s-\alpha)$ terms to get a polynomial solution. We
then have
\begin{align}
\hat{\xi}(s) &= -\sum_{k=0}^n k!f_k\Biggl(\frac{1}{\alpha^{k+1}s}
+ \frac{1}{\alpha^k s^2} + \cdots + \frac{1}{\alpha s^{k+1}} 
\Biggr) \label{eqn-4-xa}\\
&= -\sum_{k=1}^{n+1}\frac{1}{s^k} \Biggl(\frac{(k-1)!f_{k-1}}{\alpha}
+\frac{k!f_{k}}{\alpha^2}+\cdots+\frac{n!f_n}{\alpha^{n+2-k}}
\Biggr).
\label{eqn-4-xb}
\end{align}
The coefficients of $\xi(\eta)$ are evident from inspecting the
summations \eqref{eqn-4-xa} and \eqref{eqn-4-xb}. 
From the summation \eqref{eqn-4-xa} and the
inverse Laplace transform, we get
\begin{equation}
\abs{\xi} \leq  
\sum_{k=0}^n
\abs{f_k}\Biggl(\frac{k!}{0!\abs{\alpha^{k+1}}}
+\frac{k!}{1!\abs{\alpha^k}}
+ \cdots 
+\frac{k!}{(k-1)!\abs{\alpha^2}} 
+\frac{k!}{k!\abs{\alpha}}
\Biggr) = \sum_{k=0}^n
\frac{\abs{f_k}}{\abs{\alpha}}
\Biggl(\sum_{j=0}^k \frac{k!}{j!\abs{\alpha^{k-j}}}\Biggr).
\label{eqn-4-xc}
\end{equation}
To clarify the calculation that gives \eqref{eqn-4-xc}, let us
consider the special case $d\xi/d\eta = \alpha\xi + \eta^k$.  Its
unique polynomial solution is $\xi = \eta^k/\alpha - k
\eta^{k-1}/\alpha^2 -\cdots- k!/\alpha^{k+1}$ and this $\abs{\xi}$
corresponds to the 
$k$th term in \eqref{eqn-4-xc}.

Next we bound $k!/j!\abs{\alpha^{k-j}}$ for $0\leq k \leq n$ and
$0\leq j \leq k$.
\begin{align*}
\frac{k!}{j!\abs{\alpha^{k-j}}}
&= \abs{\alpha^{j-k}}k(k-1)\ldots(j+1)\\
&=\abs{\alpha^{j-k}} (k(j+1))((k-1)(j+2))
((k-2)(j+3))\ldots L\\
&\leq \Biggl(\frac{k+j+1}{2\abs{\alpha}}\Biggr)^{k-j}
\leq \Biggl(\frac{k+1/2}{\abs{\alpha}}\Biggr)^{k-j}
\leq \Biggl(\frac{n+1/2}{\abs{\alpha}}\Biggr)^{k-j}\\
& \leq 1/a^{k-j}.
\end{align*}
In the second line above, the last factor $L$ is either $(k+j+1)/2$ or
$((k+j)(k+j+2)/4)$. The first inequality in the third line is obtained
by applying the inequality $xy \leq
((x+y)/2)^2$ repeatedly. The inequality in the last line uses the
assumption $\abs{\alpha} \geq a (n+1/2)$ made in the statement of the
lemma.

Returning to \eqref{eqn-4-xc}, we have
$$
\abs{\xi} \leq \frac{\abs{f}}{\abs{\alpha}} 
(1 + 1/a + 1/a^2 + \cdots),
$$
which completes the proof.
\end{proof}

The inequality in the lemma below is not strict mainly because
$\abs{F_m}=0$ is not ruled out.

\begin{lem}
For $m \geq 8$, $\abs{X_m} \leq 192 \abs{F_m}/(m-2)$.
\label{lem-4-3}
\end{lem}
\begin{proof}
We take the matrix of eigenvectors of $A_m$ defined in 
\eqref{eqn-3-6} to be
$$
V = \begin{pmatrix}
-5i & 10 i & -5i\\
-3i/2 & i & i\\
1 & 1 & 1
\end{pmatrix},
$$
where the columns are ordered to correspond to the eigenvalues
$-m+2$, $-m$, and $-m-3$, respectively. 

If \eqref{eqn-3-5} is rewritten using a similarity transformation
that turns $A_m$ into a diagonal matrix, we get three scalar equations
$$
\frac{d\xi_{i}}{d\eta} = \alpha_{i} \xi_{i} + f_{i},
$$
for $i=1,2,3$, where $(\alpha_{1}, \alpha_{2}, \alpha_{3})
= (-m+2, -m, -m-3)$, $(f_{1}, f_{2}, f_{3})' = V^{-1} F_m$,
and $X_m = V (\xi_1, \xi_2, \xi_3)'$ (the prime denotes transpose).

To apply Lemma \ref{lem-4-2} to each of the scalar equations, we use
Lemma \ref{lem-3-1} and take $n=\floor{(m+2)/2}$. In addition, we choose
an $a>1$ such that 
$$\abs{\alpha_1} = \abs{m-2} \geq a (m+3)/2 \geq a (n+1/2).$$
The choice $a=12/11$ works for $m\geq 8$. Thus we get
$\abs{\xi_i} \leq 12\abs{f_i}/\abs{\alpha_i} \leq 12\abs{f_i}/(m-2)$
for $i=1,2,3$.

We have $\abs{f_i} \leq \norm{V^{-1}}_\infty\abs{F_m}$ for $i=1,2,3$
and $\abs{X_m} \leq \norm{V}_\infty \max(\abs{\xi}_1, \abs{\xi}_2,
\abs{\xi}_3)$.  Combining the inequalities, we get $$
\abs{X_m} \leq \frac{12}{m-2}\norm{V}_\infty \norm{V^{-1}}_\infty\abs{F_m}.
$$ 
The proof is completed by verifying that $\norm{V}_\infty \norm{V^{-1}}_\infty
=16$.
\end{proof}

The lemma below is crucial to showing that the psi series expansions
which formally satisfy the Lorenz system by Lemma \ref{lem-3-2}
are convergent. Its proof is structured to be transparent, but does not
give the best constants.

\begin{lem}
For positive constants $K_1$ and $K_2$ which depend upon the
undetermined constant $D$ of Lemma \ref{lem-3-2}, $\abs{X_m} < K_1
K_2^m$ for $m=0,1,2,\ldots$
\label{lem-4-4}
\end{lem}
\begin{proof}
By Lemmas \ref{lem-4-1} and \ref{lem-4-3}, we have
$$
\abs{X_m} \leq \frac{30\times 192}{m-2}\abs{X_{m-1}}
+\frac{28\times 192}{m-2}\abs{X_{m-2}}
+\frac{192}{m-2} \sum_{j=1}^{m-1} \abs{X_{m-j-1}}\abs{X_{j-1}}
$$
for $m\geq 8$. If we define $x_m = \abs{X_m}$
for $m=0,1,\ldots,7$ and, for $m\geq 8$,
\begin{equation}
x_m = 960 x_{m-1} + 896 x_{m-2} + 32 \sum_{j=1}^{m-1} x_{m-j-1} x_{j-1},
\label{eqn-4-3}
\end{equation}
then $\abs{X_m} \leq x_m$ (after noting $192/6 = 32$ and so on).

Let $f(Z) = \sum_{m=0}^\infty x_m Z^m$ be the generating function
of the $x_m$ sequence. Using \eqref{eqn-4-3}, we get
\begin{equation}
f(Z) - (c_0 + c_1 Z + \cdots + c_7 Z^7) = 960 Z f(Z) + 896 Z^2 f(Z)
+ 32 Z^2 f(Z)^2.
\label{eqn-4-4}
\end{equation}
In \eqref{eqn-4-4}, the constants $c_0\ldots c_7$ account for
the fact that the recurrence \eqref{eqn-4-3} is valid only for
$m\geq 8$. They are put in to get $x_0,x_1,\ldots,x_7$
as the coefficients of $Z^0,Z^1,\ldots,Z^7$,
respectively.  They can be determined explicitly (compare
Table \ref{table-1}); for instance,
$c_0= x_0 = \abs{X_0}$ and $c_1=x_1 - 960 x_0 = \abs{X_1} - 960 \abs{X_0}$.
Because $\abs{X_2},\ldots,\abs{X_7}$ depend upon $D$, so will 
$c_2,\ldots,c_7$.

The implicit function theorem implies the existence of a unique
analytic function with $f(0) = x_0$ that satisfies
\eqref{eqn-4-4}---if all terms of \eqref{eqn-4-4} are moved to the
left and `f(Z)' is treated as a variable, the partial derivative of
the left hand side with respect to $f$ is $1$ when $Z=0$, thus
verifying the derivative condition of the implicit function theorem.
Therefore $f(Z)$ is the generating function of the $x_m$ sequence. The
bound on $x_m$ given by the lemma follows from the Hadamard-Cauchy
root formula for the radius of convergence of $f(Z)$ around $Z=0$. If
$K_2$ is taken slightly greater than the inverse of the radius of
convergence and $K_1>0$, the bound $\abs{X_m} < K_1 K_2^m$ holds for
large enough $m$. So $K_1$ can be chosen to make the bound hold
for every $m=0,1,2,\ldots$

An explicit lower bound for the radius of convergence in terms of 
$c_0,\ldots,c_7$ can be determined using  the implicit
function theorem proved by Lindel\"{o}f using his majorant technique
\cite[p. 63]{Hille1976} \cite{Lindelof1899}.
\end{proof}

We are now ready to prove convergence of the formal psi series of
Lemma \ref{lem-3-2}.

\begin{thm}
Consider the formal psi series of Lemma \ref{lem-3-2} with
$\eta = \log(b(t-t_0))$ and $\abs{b} = 1$.  The branch cut is the segment
$$
\{t_0 - \bar{b} p|p\geq 0\}. $$
Then the  psi-series expansions for $x(t)$, $y(t)$, and $z(t)$ given
by
\eqref{eqn-1-2} or \eqref{eqn-3-1} converge uniformly
and absolutely on the disc $\abs{t-t_0} \leq r$ with
$r > 0$ and with an open neighborhood of the branch cut excluded
from the disc. In general, $r$ will depend upon both $C$
and $D$, which are the two undetermined constants in
the psi series.
\label{thm-4-5}
\end{thm}
\begin{proof}
We will give the proof for $z(t)$. The proofs for $x(t)$ and $y(t)$
are similar. 

Excluding a neighborhood of the branch cut means that a neighborhood
of $t_0$ is excluded from the domain of convergence. Therefore
$R_{-2}(t-t_0)^{-2}$ and $R_{-1}(t-t_0)^{-1}$ are both bounded on
the domain of convergence. The other reason for excluding
a neighborhood of the branch cut is to ensure that $\eta$ is
well-defined.

By the definitions of $\abs{R_m}$ and $\abs{X_m}$ given at the
beginning of this section,
\begin{align*}
\abs{R_m(\eta)(t-t_0)^m} 
&\leq \abs{R_m} \max\Bigr(1, \Abs{\log b(t-t_0)+C}^{\floor{(m+2)/2}}\Bigl) \abs{t-t_0}^m\\
&\leq \abs{X_m} \max\Bigr(1, \Abs{\log b(t-t_0)+C}^{\floor{(m+2)/2}}\Bigr) \abs{t-t_0}^m\\
&< K_1 K_2^m \max\Bigl(1, \Abs{\log b(t-t_0)+C}^{\floor{(m+2)/2}}\Bigr) \abs{t-t_0}^m,
\end{align*}
where $m\geq 0$ and $\abs{b}=1$. The first inequality above uses Lemma \ref{lem-3-1}
and the third inequality uses Lemma \ref{lem-4-4}.

Choosing an $r>0$ such that
\begin{equation}
r < 1/K_2 \quad\quad \text{and} \quad\quad 
r(\abs{\log r} + \pi + \abs{C}) < 1/K_2
\label{eqn-4-5}
\end{equation}
is sufficient to ensure uniform and absolute convergence. The $\pi$ in
\eqref{eqn-4-5} is explained by the inequality $\abs{\log b(t-t_0)}
\leq \abs{\log\abs{t-t_0}} + \pi$. A choice of $r$ in 
accord with \eqref{eqn-4-5} suffices for the convergence of the psi
series for $x(t)$ and $y(t)$ as well.
\end{proof}

A further argument is required to show that 
the convergent psi series actually satisfy the
differential equation, a point that seems to have been overlooked on
occasion. When the psi series for $x(t)$, $y(t)$ and
$z(t)$ are substituted into the Lorenz system \eqref{eqn-1-1}, the
summation and multiplication of psi series on the right hand side is
justified by standard results on rearrangements of absolutely
convergent series.  To justify the differentiation of psi series on
the left hand side, we mention that the uniform convergence of a sequence of
analytic functions on an open set implies the uniform convergence of
the derivatives on any compact subset of that open set \cite[Theorem
10.28]{Rudin1976}.  We can now state the following theorem.

\begin{thm}
The psi series for $x(t)$, $y(t)$ and $z(t)$ given by \eqref{eqn-1-2}
or \eqref{eqn-3-1}, whose formal validity is asserted by 
Lemma \ref{lem-3-2}, satisfy the Lorenz system \eqref{eqn-1-1}
in the disc $\abs{t-t_0} \leq r$, with the
branch cut excluded, for some $r>0$. In general, $r$ will depend upon both $C$
and $D$, which are the two undetermined constants in
the psi series.
\label{thm-4-6}
\end{thm}

Levine and Tabor \cite{LevineTabor1988} raised the possibility
that the locations of the singularities of an orbit of the Lorenz
system may have accumulation points in the complex $t$-plane. Theorem
\ref{thm-4-6} shows that psi series singularities cannot be 
accumulation points.

So far, in results such as Lemma \ref{lem-3-2} and Theorem
\ref{thm-4-6}, we have regarded the psi series as functions of $t$. It
is useful to consider them as functions of $\eta$, where $\eta = \log
b(t-t_0)$ gives a parametrization of the Riemann surface that
gets rid of the branch cut in the $t$-plane.  To be specific, we
assume $\Im(t_0) < 0$ and $b=-i$. In that case, we have $(t-t_0) = i \exp(\eta)$ and the
psi series
\eqref{eqn-3-1} take on the form
\begin{equation}
x(\eta) = \sum_{m=-1}^{\infty} i^m P_m(\eta)e^{m\eta} \quad
y(\eta) = \sum_{m=-2}^{\infty} i^m Q_m(\eta)e^{m\eta} \quad 
z(\eta) = \sum_{m=-2}^{\infty} i^m R_m(\eta)e^{m\eta},
\label{eqn-4-6}
\end{equation}
with $P_m,\, Q_m,\, R_m$ being polynomials in which $\eta$ always
occurs in the group $(\eta + C)$. Every time $t$ passes through the
branch cut of $\log(-i(t-t_0))$, $\eta$ increases or decreases
by $2\pi i$. Because $\eta$ and $C$ always occur in the group
$(\eta+C)$ in $P_m,\, Q_m,\, R_m$, we can allow for other branches of
$\log(-i(t-t_0))$ in the psi series of \eqref{eqn-3-1} or \eqref{eqn-1-2} by
keeping the principal branch of the logarithm in the
definition of $\eta$ and incrementing $C$ by
an integer multiple of $2\pi i$. 
The change in the estimate for the
radius of convergence $r$ of Theorem \ref{thm-4-6} for these other
branches will then be in accord with \eqref{eqn-4-5} (note that $K_2$
depends only on $D$).

\begin{figure}
\begin{center}
\begin{pspicture}(-5,-3)(5,3)
\psset{xunit=1cm,yunit=1cm}
\psline[linewidth=2pt,arrows=->](-5,0)(5,0)
\psline[linewidth=2pt,arrows=->](0,-3)(0,3)
\rput[br](4.75,0.25){$\Re(\eta)$}
\rput[tl](0.25,2.75){$\Im(\eta)$}
\psframe[linestyle=none,fillstyle=crosshatch](-5,-0.5)(-0.1,0.5)
\qline(-5,-0.5)(-0.1,-0.5)
\qline(-0.1,-0.5)(-0.1,0.5)
\qline(-0.1,0.5)(-5,0.5)
\psframe[linestyle=none,fillstyle=crosshatch](-5,0.5)(-2.5,1.5)
\qline(-5,0.5)(-2.5,0.5)
\qline(-2.5,0.5)(-2.5,1.5)
\qline(-2.5,1.5)(-5,1.5)
\psframe[linestyle=none,fillstyle=crosshatch](-5,1.5)(-3.25,2.5)
\qline(-5,1.5)(-3.25,1.5)
\qline(-3.25,1.5)(-3.25,2.5)
\qline(-3.25,2.5)(-5,2.5)
\psframe[linestyle=none,fillstyle=crosshatch](-5,-0.5)(-2.6,-1.5)
\qline(-5,-0.5)(-2.6,-0.5)
\qline(-2.6,-0.5)(-2.6,-1.5)
\qline(-2.6,-1.5)(-5,-1.5)
\psframe[linestyle=none,fillstyle=crosshatch](-5,-1.5)(-3.4,-2.5)
\qline(-5,-1.5)(-3.4,-1.5)
\qline(-3.4,-1.5)(-3.4,-2.5)
\qline(-3.4,-2.5)(-5,-2.5)
\end{pspicture}
\end{center}
\caption[xyz]{Schematic plot of domain of convergence of the psi series in the $\eta$ plane, as implied by Theorem \ref{thm-4-5}.
The shape of the region is given approximately by 
$\Re(\eta) \precsim -\log\abs{m}$ for
integers $m$ of large magnitude and 
$-\pi + 2\pi m < \Im(\eta) \leq \pi + 2\pi m$.
}
\label{fig-4}
\end{figure}
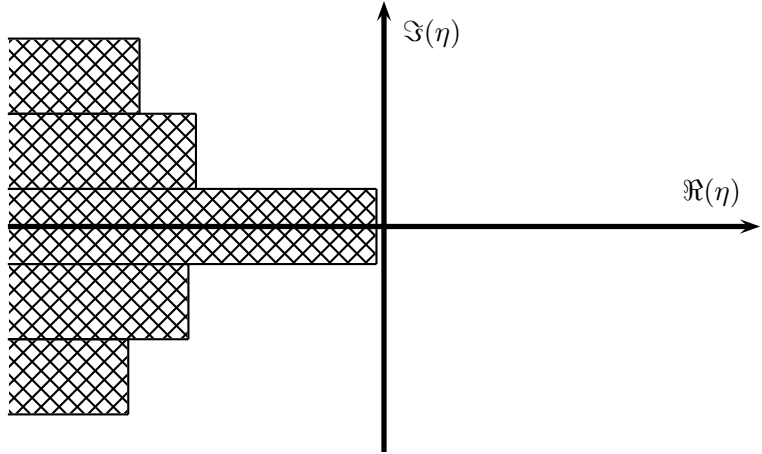

If the domain of convergence of the transformed psi series
\eqref{eqn-4-6} is considered in the $\eta$-plane, the choice
of the principal branch of $\log(-i(t-t_0))$ 
implies $-\pi < \Im(\eta) \leq \pi$ and the
$r$ estimated by Theorem \ref{thm-4-6} implies $\Re(\eta) \leq \log
r$. Thus the region of convergence of the principal branch will be a
semi-infinite rectangle in the $\eta$-plane.  To pass to other
branches, we keep $C$ fixed and allow the imaginary part of $\eta$ to
be arbitrary. For $\eta$ corresponding to different branches, one has
to use different estimates for $r$ as explained in the previous
paragraph. Therefore the estimated domain of convergence of the
transformed psi series \eqref{eqn-4-6} will be a union of
semi-infinite rectangles as in Figure \ref{fig-4}. If we start at the
principal branch of $\log(-i(t-t_0))$
and cross its branch cut $m$ times, then by
\eqref{eqn-4-5} $r \approx 1/(K_2 2\pi\abs{m})$ for large integers $m$. For
such a branch $-\pi + 2\pi m < \Im(\eta) \leq \pi + 2\pi m$ and
$\Re(\eta) \precsim -\log\abs{m}$ for convergence, which gives an
approximate idea of the shape of the domain sketched in Figure
\ref{fig-4}.

\subsection{Remarks on theorems of Hille and Smith}

In \cite{Hille1974}, Hille proved that the plane quadratic system
\begin{align*}
dx/dt &= x(a_0 + a_1 x + a_2 y)\\
dy/dt &= y(b_0 + b_1 x + b_2 y)
\end{align*}
has a logarithmic psi series singularity if
$(a_1-b_1)(a_2-b_2)/(a_1b_2-a_2b_1)$ is a positive integer. Smith
\cite{Smith1975} generalized that result to plane polynomial systems.
Smith's proof is based on a reduction to results proved early in the
20th century for Briot-Bouquet systems.  These results are summarized
in Sections 12.5 and 12.6 of Hille's book \cite{Hille1976}.

One difference between the results of Hille and Smith for plane
polynomial systems and Theorem \ref{thm-4-6} is as follows. The
singular solutions for plane polynomial systems look like simple poles
near the singular point.  The singularities of the Lorenz system
implied by Theorem \ref{thm-4-6} look like double poles.

In \cite{Hille1973}, Hille proved the existence of logarithmic psi
series solutions for the Emden-Fowler system $d^2 y/dt^2 = t^{-2/p}
y^{1+2/p}$ for $p=2$. At the end of the paper, Hille discussed the
difficulty of extending his technique and noted remarks by a referee
suggesting a proof of existence of logarithmic psi series solutions
for positive integral $2p$. Like Smith's proof for plane polynomial
systems, the suggested proof goes through a reduction to a
Briot-Bouquet system, but no complete proofs are found in the
literature as far as we are aware. The result for positive integral
$2p$ was stated as Theorem 12.4.2 in Hille's book
\cite{Hille1976}. Hille mentioned that ``the various proofs are nasty,''
while omitting them.

The proofs using reduction to Briot-Bouquet systems are difficult to
follow in their entirety, partly because they depend so crucially on
results proved long ago. It appears that use of the Laplace transform
and the implicit function theorem will give simpler proofs for
plane polynomial systems and complete proofs that are not so nasty in
the case of the Emden-Fowler system with $2p$ a positive integer.

Theorem 4 of Smith's paper \cite{Smith1975} states that all singularities
of real solutions of certain plane polynomial systems must be of
the form determined in Theorem 3 of that paper. The statement occurs
again as Theorem 12.6.3 of \cite{Hille1976}. Smith's proof begins
with an ingenious change of variables. Near the end of the proof, we find
the argument ``in the case when $\lambda > 0$ , the arbitrary
constant $c$ in (20) can be chosen to fit this solution
$\zeta(\xi)$ in the neighborhood of $\xi = 0$.'' We are unable
to follow that argument and believe it requires substantial
explication at the very least.

\section{Complex singularities and the Lorenz attractor}

If $t=t_0$ is a singularity of the Lorenz system \eqref{eqn-1-1},
the solution must diverge to infinity as the singularity is
approached.

\begin{thm}
Let $\gamma$ be a Lipshitz curve in the complex $t$ plane that
approaches $t_0$ at one of its two endpoints.  Let $(x(t), y(t),
z(t))$ be a solution of the Lorenz system
\eqref{eqn-1-1} defined
for $t\in \gamma$. If $t_0$
is a singular point, then
\begin{equation}
\liminf_{t\rightarrow t_0} \abs{t-t_0} 
\Bigl(\abs{x(t)} + \abs{y(t)} + \abs{z(t)}\Bigr)
\geq \frac{1}{8},
\end{equation}
as $t$ approaches $t_0$ along the curve $\gamma$.
\label{thm-5-1}
\end{thm}
\begin{proof}
Denote $\abs{x(t)} + \abs{y(t)} + \abs{z(t)}$ by $r_t$.  Consider the
set of all complex $(x,y,z)$ in the region $\abs{x-x(t)} +
\abs{y-y(t)} +
\abs{z-z(t)} < b$ for some $b > 0$. Then the sum of the absolute values
of the right hand sides of the Lorenz system \eqref{eqn-1-1} is bounded
by
$$
M=52(r_t + b) + 2 (r_t+b)^2,
$$
where $10+10+28+1+8/3 < 52$ explains the first coefficient.

Theorem 8.1, Chapter 1, of \cite{CoddingtonLevinson1955} (also
see Theorem 2.3.1 of \cite{Hille1976}) with $a=\infty$ and
$M$ and $b$ as above implies that the solution admits a unique
analytic continuation to all $t'$ in the disc $\abs{t'-t} \leq R$
with 
$$
R = \frac{b}{52(r_t + b) + 2 (r_t+b)^2}.
$$ 
Taking $b=r_t$, we get $R=1/(104+8 r_t)$.

Being a singular point, $t_0$ must lie outside the disc of
analyticity. Therefore $\abs{t_0-t} (104 + 8 r_t) > 1$.
Taking the limit $t\rightarrow t_0$ along points on $\gamma$
completes the proof.

The curve $\gamma$ is assumed to be Lipshitz to ensure uniqueness
of the solution.
\end{proof}

Theorem \ref{thm-5-1} proves that as the singular point $t_0$ of the
Lorenz system is approached, the magnitude of the solution must
diverge at a rate that is at least as great as $0.125/\abs{t-t_0}$.
In fact, if the answer to Question \ref{qstn-1} is yes and the
singularities of the Lorenz system are all given by psi series of the
form
\eqref{eqn-1-2},  the divergence would be proportional to
$1/\abs{t-t_0}^2$.

Theorem \ref{thm-5-1} is used to prove the theorem below.

\begin{thm}
Consider a trajectory of the Lorenz system \eqref{eqn-1-1}
which is real for real values of $t$.  In particular, assume that the
state $(x(0), y(0), z(0))$ at $t=0$ is real.  Then there is no
singularity at any finite and real value of $t$ and the solution is
defined for all real values of $t$.
\label{thm-5-2}
\end{thm}
\begin{proof}
Let $Q = x^2 + y^2 + z^2$. From \eqref{eqn-1-1}, we have
$$ dQ/dt =
2(-10 x^2 - y^2 - 8z^2/3 + 38 xy). 
$$
The matrix $1$-norm of the symmetric form on the right hand side is bounded by
$58$ and so are the magnitudes of its eigenvalues. Therefore,
$\abs{dQ/dt} < 58 Q$ and 
$$
Q(t) \leq Q(0) + 58\int_0^{\abs{t}}Q(s) ds.
$$
At this point it appears as if the proof can be completed using
the Gronwall inequality (Theorem 1.6.6 of \cite{Hille1976}) to deduce
that $Q(t) \leq Q(0)\exp(58\abs{t})$. However, the bound on $Q(t)$
holds only if we assume the existence of the solution, which is what
we set out to prove.

In circumstances such as these, oscillatory singularities for which
the solution does not tend to a limit as the singular point is approached
must be ruled out --- an important point that goes back to
Painlev\'{e}
\cite[Chapter 3]{Hille1976}. Theorem \ref{thm-5-1} forces the
norm of the solution to diverge near a singular point thus making it
possible to complete the proof.
\end{proof}

Theorem \ref{thm-5-2} is implied by Theorem 2.4, part $(i)$, of
\cite{FoiasJollyKukavicaTiti2001}. 
In fact, Theorem 2.4 of \cite{FoiasJollyKukavicaTiti2001} is a sharper
result as it implies that $Q(t) \leq C \exp(20 \abs{t})$ for some
constant $C$ independent of $t$.  We have given a proof that brings
out the connection to the nature of the singular points.

\begin{table}
\begin{center}
\begin{tabular}{c|c}
$AB$ & $\pm i0.1714501006$ \\ \hline $AAB$ & $\pm i0.1617621257$ \\
\hline $AAAB$ & $\pm i0.1563426260$ \\ \hline $AABB$ & $\pm
i0.1636066901$
\end{tabular}
\end{center}
\caption[xyz]{Imaginary parts of the singular points
closest to the real line in the $t$-plane.}
\label{table-2}
\end{table}

So far we know that the Lorenz system has singularities represented by
logarithmic psi series and that the solution must diverge as a
singularity is approached. But do solutions such as the one shown in
Figure \ref{fig-1} have complex singularities and are they represented
by psi series?

Using numerical methods based on \cite{PaulsFrisch2007}, we found the
complex singularities closest to the real line of a few solutions
listed in Table \ref{table-2}.  Those solutions are all of course real
for real $t$. They are assigned the labels $AB$, $AAB$, $AAAB$ and
$AABB$ following the convention explained in the caption to Figure
\ref{fig-1}. From Table \ref{table-2}, we see that the complex
singularities are located at a distance greater than $0.037$ from the
real line, in agreement with Theorem 2.3 of Foais et al.\hspace{-.1cm}
\cite{FoiasJollyKukavicaTiti2001}. In addition to computing the location
of the singularities, we have verified numerically that their form
matches the formal development of psi series given in Section 3. This
numerical work will be described in detail elsewhere.

\section{Conclusion}
Given that the Lorenz system \eqref{eqn-1-1} has resisted mathematical
analysis on the real line, one may say that it is natural to think of
$t$ as a complex variable and $x,y,z$ as analytic functions of $t$.
When the solutions of the Lorenz system are viewed as analytic
functions, it is natural to begin their investigation by looking at
their singularities.  We have given a complete formal development of
singularities in the complex $t$-plane, proved convergence of the psi
series representations using a new technique, and proved that the psi
series indeed satisfy the Lorenz system. 
The development of the analytic theory
appears to be a fascinating avenue for further investigations.

Our suggestion that the mathematical analysis of the Lorenz system 
\eqref{eqn-1-1} could be a problem in analytic function theory
is an attempt to complete the circle,
because the geometrical theory of differential
equations, in which the Lorenz system is a famous example, sprang out
of problems in analytic function theory---a fact that is not too
well-known.
More specifically, the stable
manifold theorem, which is undoubtedly fundamental to the geometrical
theory, was first proved to understand the solution of $dz/dw =
P(z,w)/Q(z,w)$ in a neighborhood of $z=w=0$ when $P$ and $Q$ are
bivariate polynomials with $P(0,0)=Q(0,0)=0$
\cite[p. 97]{Hille1976}\cite[1880]{Poincare1880}.

The properties of analytic functions $x(t)$ which satisfy the
nonlinear Riccati equation $dx/dt = f_0(t) + f_1(t)x + f_2(t) x^2$,
where the $f_i(t)$ are rational in $t$, is a well-studied topic.  All
the movable singularities of the Riccati equation are poles and the
dependence of its solution on the undetermined constant is given by a
fractional linear transformation. For the Lorenz system some of the
movable singularities have psi series representations of the form
determined in Section 4. The dependence of these psi series solutions
on the undetermined constants is much more complicated than for the
Riccati equation.

Another well-studied topic is the classification of
second order nonlinear systems all of whose movable singularities are
poles.  The Painlev\'{e} classification has been presented with
lexicographic thoroughness by Ince \cite{Ince1956}. There appear to be
few classification results for third order systems such as the Lorenz
system. Studying a specific system will probably sidestep many
difficulties of the classification problem. In any event, the movable
singularities of the Lorenz system are not poles.

\section{Acknowledgments}
Many thanks to one of the referees for finding an error in Section 4
and for graciously allowing us to correct it.  Thanks to Jeff
Lagarias for valuable pointers related to content and exposition, and
for generously sharing his notes on the Lorenz system.  
Thanks to Nick Trefethen for a number of valuable suggestions and
comments, some of which we hope to address fully later. Thanks to
B. Deconinck, B. Eckhardt, R. Goodman, A. Iserles and M. Slemrod for
helpful discussions. DV thanks the Mathematics Department, Indian
Institute of Science, and the Computing Laboratory, Oxford University,
for their hospitality.

\bibliography{references} \bibliographystyle{siam}

\end{document}